\documentclass[11pt]{amsart}
\usepackage{hyperref}
\hypersetup{
    colorlinks=true,
    linkcolor=blue,
    citecolor=blue,
    urlcolor=blue,
    pdfauthor={Diego Corro and Adam Moreno},
    pdftitle={Core Reduction for Singular Riemannian Foliations in Positive Curvature}
}

\usepackage[numbers]{natbib}
\usepackage{url}
\usepackage{graphicx}
\usepackage{amsfonts, amssymb, amsmath, amsthm}
\usepackage{mathtools}
\usepackage{multicol}
\usepackage{tikz}
\usepackage{tikz-cd}
\usepackage{tikz-3dplot}
\usepackage{cancel}
\usepackage{changepage}
\usepackage{csquotes}
\usepackage{adjustbox}
\usetikzlibrary{matrix}
\usepackage{enumerate}
\usepackage[normalem]{ulem}

\newtheorem{mtheorem}{Theorem}  
 
\newtheorem{mcorollary}[mtheorem]{Corollary}

\newtheorem{theorem}{Theorem}[section]
\newtheorem{lemma}[theorem]{Lemma}
\newtheorem{prop}[theorem]{Proposition}

\theoremstyle{definition}
\newtheorem{definition}[theorem]{Definition}
\newtheorem*{ack}{Acknowledgements}

\newtheorem*{ex}{Example}
\newtheorem{rmk}[theorem]{Remark}

\newcommand{\FF}{\mathcal{F}}
\newcommand{\bdy}{\partial}
\newcommand{\red}[1]{\textcolor{red}{#1}}
\newcommand{\blue}[1]{\textcolor{blue}{#1}}
\newcommand{\magenta}[1]{\textcolor{magenta}{#1}}
\newcommand{\pre}[1]{\prescript{}{#1}}
\newcommand{\G}[1]{\prescript{}{#1}G}
\newcommand{\hol}{\Gamma}

\newcommand{\wtilde}{\widetilde}

\newcommand{\Sp}{\mathbb{S}} 

\DeclareMathOperator{\Iso}{Isom}

\DeclareMathOperator{\codim}{codim}

\DeclareMathOperator{\reg}{reg}

\DeclareMathOperator{\expo}{exp}
\DeclareMathOperator{\cl}{Cl}

\usepackage{lineno}

\usepackage[nostamp]
{draftwatermark}
\SetWatermarkScale{5}

\begin{document}
\title{Core Reduction for Singular Riemannian Foliations in Positive Curvature}

\author{Diego Corro$^{*\dagger}$}
\address{Instituto de Matem\'{a}ticas, Unidad Oaxaca\\
Universidad Nacional Aut\'{o}noma de M\'{e}xico (UNAM)\\
Antonio de Le\'{o}n \#2, altos, Col. Centro\\
Oaxaca de Ju\'{a}rez, Oaxaca, CP. 68000\\
M\'{e}xico.
	   }
\email{diego.corro.math@gmail.com}
\thanks{$^{*}$Supported by the DFG (281869850, RTG 2229 ``Asymptotic Invariants and Limits of Groups and Spaces'').}
\thanks{$^{\dagger}$Supported by a DGAPA postdoctoral Scholarship of the Institute of Mathematics - UNAM}

\author{Adam Moreno}
\address{Department of Mathematics\\
University of California, Los Angeles\\
      Los Angeles, CA, USA 90095-1555
      }
\email{moreno@math.ucla.edu}


\setlength{\overfullrule}{5pt}

\begin{abstract}
We show that for a smooth manifold equipped with a singular Riemannian foliation, if  the foliated metric has positive sectional curvature, and  there exists a pre-section, that is a proper submanifold retaining all the transverse geometric information of the foliation, then the leaf space has boundary. In particular, we see that polar foliations of positively curved manifolds have leaf spaces with nonempty boundary.
\end{abstract}

\maketitle

\section*{Introduction}
All known examples of positively curved Riemannian manifolds have, in some sense or another, `large' symmetry. This observation led to the initiation of the Grove symmetry program in 1991, birthing several systematic approaches to explore the link between the isometries of positively curved manifolds and their topology. Many techniques used rely not so much on the particular subgroup of isometries considered, but on how the orbits of those isometries decompose the given manifold. Such orbit decompositions are special cases of what are more generally known as \textit{singular Riemannian foliations}. Galaz-Garcia and Radeschi used this more general framework to study positively curved manifolds carrying such foliations whose regular leaves are tori (see \cite{galaz2015singular}), generalizing some known results for torus actions as well as pointing out some interesting differences. Mendes and Radeschi, Corro, and Moreno (\cite{MendesRadeschi2019}, \cite{Corro2019}, and \cite{moreno2019}) have continued this approach, establishing singular Riemannian foliations as a possible notion of symmetry for positively curved manifolds.

When the leaves of the foliation are closed, the so-called \textit{leaf space} of the foliation is equipped with a natural metric which inherits lower curvature bounds in the comparison sense (see \cite{BBI}). This allows one to study such quotients using Alexandrov geometry. In particular, the notions of spaces of directions and boundary are easy to describe in terms of the foliation and can hence be employed to study singular Riemannian foliations and the manifolds which admit them. For example, Grove, Moreno, and Petersen showed in \cite{GroveMorenoPetersen2018} that the boundary of a leaf space is an Alexandrov space with the same lower curvature bound (in its induced intrinsic metric), answering an open question about Alexandrov spaces in this special case.

For positively curved leaf spaces, the presence of boundary already places topological restrictions on both the leaf space and the manifold (see \cite{morenophd}, Theorem~4.2.3). Moreover, placing some simple additional hypotheses on the topology of the boundary can yield strong topological implications on not just the leaf space and manifold, but also the leaves of the foliation  (see \cite{morenophd}, Theorem~4.3.1). Taking a step back, one is inclined to ask: what are some sufficient conditions that guarantee the presence of nonempty boundary?

In \cite{wilking2006positively}, Wilking showed that for positively curved manifolds, an isometric action with nontrivial principal isotropy group will have orbit space with nonempty boundary. This does not immediately generalize to leaf spaces of singular Riemannian foliations, where there is no group action. However, nontrivial principal isotropy guarantees a nontrivial \textit{core reduction} as defined for group actions by Grove and Searle in \cite{grovesearlecore}.

The core $\prescript{}{c}{M}\subset M$ of an isometric group action $(M,G)$ (together with its \textit{core group} $\prescript{}{c}{G}$) form a `reduction' of the group action. In particular, $\prescript{}{c}{G}<G$ and the orbit spaces $M/G$ and $\prescript{}{c}{M}/\prescript{}{c}{G}$ are isometric.  A related, yet more extreme, form of reduction is given by polar actions, where we have a (connected, complete) embedded submanifold $\Sigma$ (called a \textit{section}) with a finite action by a so-called \textit{polar group} $W$ (also called the \textit{Weyl group}) such that $M/G$ and $\Sigma/W$ are isometric. Despite the geometric similarities, a section of a polar action is not necessarily a \textit{core} of that action. A simple example is the $S^1$ action on $S^2$ by rotation about a fixed axis. This is a polar action, whose section is a great circle and whose polar group is $\mathbb{Z}_2$. On the other hand, the core is all of $S^2$ (see \cite{grovesearlecore} for definitions).

In \cite{gorodski2004copolarity}, the authors introduced the notion of \textit{copolarity} to define a \textit{k-section} of an isometric group action, providing a broader framework to discuss such reductions. In this language, a section of a polar action is at one extreme: it is a $0$-section, while the core of a group action $(M,G)$ is an example of a $k$-section, where $k$ is the difference between the dimension of the core and the dimension of the orbit space $M/G$. The core in the example above (the entire original manifold) is an example of a $1$-section. In \cite{magata2009reductions}, Magata referred to $k$-sections of group actions as \textit{fat sections} and proved that they are a form of reduction in the sense above (see Thm 3.1 \cite{magata2009reductions}). 

Clearly, these reductions rely on the presence of a group action. \textit{Polar foliations} generalize polar actions to the setting of singular Riemannian foliations by using the geometric properties of sections of polar actions to define a section of a foliation. This is more than simply an expanding of language, as there are polar foliations whose leaves are not the orbits of an isometric group action (most FKM-type foliations by isoparametric hypersurfaces, see \cite{Radeschi2014}). This sort of reduction is extreme in the foliation setting as well. An expansion of this notion (similar to what was done in \cite{gorodski2004copolarity}) was proposed in the thesis of Magata \cite{magata2008reductions} (where he referred to them as \textit{pre-sections}), though was not developed beyond a definition. 

Motivated by the guarantee of boundary in the presence of a nontrivial core and the more general framework of pre-sections to which cores belong, we prove the following: 

\begin{mtheorem}\label{Theorem: Core reduction} Let $\FF$ be a singular Riemannian foliation with closed leaves on a positively curved manifold $M$. If $(M,\FF)$ has a nontrivial pre-section, then $\partial(M/\FF)\neq \emptyset$.
\end{mtheorem}

Since sections of polar foliations are a special type of pre-section, we have:

\begin{mcorollary}
If $(M,\FF)$ is a closed, polar singular Riemannian foliation (that is not a single leaf) on a positively curved manifold, then $M/\FF$ has nonempty boundary.
\end{mcorollary}

We also get the following special case regarding orbit spaces of isometric group actions:

\begin{mcorollary}
Let $G$ be a compact lie group acting isometrically and not transitively on a positively curved Riemannian manifold $M$. If $M$ contains a $k$-section, then $M/G$ has nonempty boundary.
\end{mcorollary}

\begin{ack} We would like to thank Fernando Galaz-Garcia, Karsten Grove, Alexander Lytchak and Marco Radeschi for helpful conversations. The second author thanks the hospitality of the Department of Mathematics of the KIT, where part of the present work was carried. 
\end{ack}


\section{Preliminaries}\label{S: Preliminaries}


In this section we  present some preliminary definitions and results, which will be used later. We begin with some  definitions about singular Riemannian foliations, and  then we give a local description of singular Riemannian foliations with closed leaves.

\subsection{Singular Riemannian foliations}\label{SRF}
\ \\
Given a Riemannian manifold $M$, a \emph{singular Riemannian foliation}, which we denote by $(M,\FF)$, is a partition of $M$ by  a collection $\FF = \{L_p\mid p\in M\}$ of  connected, complete, immersed submanifolds $L_p$, called \emph{leaves}, which may not be of the same dimension, such that the following conditions hold: 
\begin{enumerate}[(i)]
\item Every geodesic meeting one leaf perpendicularly, stays perpendicular to all the leaves it meets.
\item For each point $p\in M$ 
there exists a family of vectors fields on $M$ which at any point $p\in M$, span the tangent space to the leaf through $p$.
\end{enumerate}

If the partition $(M,\FF)$ satisfies the first condition, then we say that $(M,\FF)$ is a \emph{transnormal system}. If it satisfies the second condition, we say that $(M,\FF)$ is a \emph{smooth singular foliation}. When all the leaves have the same dimension, we say that the foliation is a \emph{regular Riemannian foliation} or just a \emph{Riemannian foliation}. We will deal mainly with \textit{closed} singular Riemannian foliations - those in which every leaf is closed (compact without boundary). The interested reader can consult \cite{Alexandrino2012, Corro2019, moreno2019} for a more  detailed discussion of singular Riemannian foliations.

A standard example of a singular Riemannian foliation is the orbit decomposition of a Riemannian manifold under some group action by isometries. Such foliations are called \textit{homogeneous}, in reference to their leaves being homogeneous manifolds. Although the geometry of singular Riemannian foliations closely resembles that of orbit decompositions, there are important examples of singular Riemannian foliations which do not come from isometric group actions. The fibers of a Riemannian submersion also provide examples of singular Riemannian foliations, and a well known \textit{inhomogeneous} such foliation is given by the $S^7$ fibers of the Hopf map $S^{15}\to S^8$ (see \cite{Radeschi2014} for more examples).

For a connected manifold $M$, the \emph{dimension} of a  foliation $\FF$, denoted by $\dim \FF$, is the maximal dimension of the leaves of $\FF$. The \emph{codimension} of a foliation is,
\[
	\codim(M,\FF) = \dim M -\dim \FF.
\] 
Leaves of maximal dimension are called \emph{regular leaves} and the remaining 
leaves are called \emph{singular leaves}.
Since $\FF $ gives a partition of $M$, for each point $p\in M$ there is a unique leaf, which we denote by $L_p$, that contains $p$. We say that $L_p$ is the \emph{leaf through} $p$.\\

The quotient space $M/ \FF$ obtained from the partition of $M$, is known as the \emph{leaf space} and the quotient map $\pi \colon M \to M/\FF$ is the \emph{leaf projection map}. The topology of $M$ yields a topology on $M/\FF$, namely the quotient topology. With respect to this topology the quotient map is continuous. We denote the leaf space $M/\FF$ from this point onward by $M^\ast$ and denote by $S^\ast$ the image $\pi(S)$ of a subset $S\subset M$ under the leaf projection map. 

Given a singular Riemannian foliation $(M,\FF)$, we denote by $O(M,\FF)$ the group of isometries of $M$ which respect the foliation (map leaves to leaves) and by $O(\FF)$ the group of isometries which leaves any leaf of the foliation invariant. Observe that $O(M,\FF)/O(\FF)$ is the group of bijections of the leaf space $M/\FF$ which lift to isometries of $M$ (see \cite{MendesRadeschi2019}).

Let $(M,\FF)$ be a singular Riemannian foliation. A piecewise smooth curve $c$ is called \textit{horizontal} with respect to  the foliation $\FF$, if $c'(t)$ is in the normal space $\nu_{c(t)}(L_{c(t)})$ of the leaf $L_{c(t)}$ at $c(t)$. The \textit{dual foliation} $(M,\FF^{\#})$ of $\FF$ is given by defining for a point $p\in M$ the leaf as
\[
	L^{\#}_p = \{q\in M\mid \mbox{there is a piecewise smooth horizontal curve from } p \mbox{ to } q\}.
\]

The following is a main result from \cite{Wilking2007}  which will be central later on.

\begin{theorem}[Theorem~1 in \cite{Wilking2007}]\label{T: Wilking one leaf}
Suppose that $M$ is a complete positively curved manifold with a singular Riemannian foliation $\FF$. Then the dual foliation has only one leaf, $M$.
\end{theorem}
\vspace{.5cm}


\subsection{Infinitesimal foliations}\label{Ss: Infinitesimal foliation}
\ \\
Let $(M,\FF)$ be a closed manifold with a closed singular Riemannian foliation. In this section we describe the foliation around a  point, and in a tubular neighborhood of a leaf.

For fixed $p\in M$, let $\Sp^\perp_p$ be the unit sphere in the normal subspace of $T_p L_p\subset T_pM$. The \textit{infinitesimal foliation} $\FF_p$ on $\Sp^\perp_p$ is given by taking the connected components of  the preimages  under the exponential map at $p$  of the intersection between the leaves of $\FF$ and $\exp_p(\Sp^\perp_p)$. By \cite[Proposition~6.5]{Molino} this partition is a singular Riemannian foliation when we consider the round metric on $\Sp^\perp_p$. Traditionally, an \textit{infinitesimal foliation} $(V,\FF)$ refers to a singular Riemannian foliation of a Euclidean space $V$ containing the origin as a leaf. Since any such foliation is the cone of the foliation on the unit sphere in $V$, we use this term to refer to both foliations (the context will make it clear which is meant).

There is a group morphism $\rho_p\colon \pi_1(L_p,p)\to O(\Sp^\perp_p,\FF_p)/O(\FF_p)$. Denote by $\hol_{p}$ the image of $\pi_1(L_p,p)$ under this morphism. The group $\hol_{p}$ is known as the \emph{leaf holonomy group  of $L_p$}. In particular, the group $\Gamma_p$ acts effectively by isometries on the leaf space $\Sp_p^{\perp}/\FF_p$. 
The interested reader can consult \cite[Section~3.2]{MendesRadeschi2019} or \cite[Appendix~A]{Corro2019} for more details on the definition of this group morphism. Leaves of maximal dimension with $\hol_{p}$ equal to the trivial group are called \emph{principal leaves} (see \cite{Corro2019}).

\subsection{Alexandrov Geometry of Leaf Spaces}
\ \\
We briefly mention some concepts from Alexandrov geometry that we will later need.

A locally compact, locally complete inner metric space is an \emph{Alexandrov space $(X,d)$} if it satisfies local lower curvature bounds as in Topogonov’s Theorem (see \cite{BGP}). In the case that the lower curvature bound is $k$, we write either $curv(X)\geq k$ or $X\in Alex(k)$.

The \emph{dimension} of an Alexandrov space $X$ is equal to the Hausdorff dimension of $X$. In particular for $(M,\FF)$ a closed singular Riemannian foliation on a complete connected manifold, the dimension of $M^\ast$ equals the codimension of $\FF$.

Without tangent spaces, Alexandrov spaces do not have the usual `tangent sphere' as manifolds do. Instead, one describes an analogous concept, the so called \textit{space of directions}, using the metric and comparisons to a model space as follows.

Consider $X\in Alex(k)$. Given two curves $c_1\colon [0,1]\to X$ and $c_2\colon [0,1]\to X$ with $c_1(0) = c_2(0) = x\in X$, we define the \emph{angle between $c_1$ and $c_2$} as
\[
	\angle (c_1,c_2) := \lim_{s,t\to 0} \tilde\angle (c_1(s),x,c_2(t)).
\]
where $\tilde{\angle} (c_1(s),x,c_2(t))$ is the angle in the comparison triangle in the appropriate model space of constant curvature $k$.

\begin{rmk}
Given a closed singular Riemannian foliation $(M,\FF)$, the leaf projection map $\pi:M\to M^*$ is an example of a \textit{submetry} - a map between metric spaces which takes metric balls to metric balls of equal radius (i.e. $\pi(B_r(x))=B_r(\pi(x))$). Because of this, any lower curvature bound of $M$ (in the comparison sense) is inherited by $M^*$. Hence, if $(M,\FF)$ is a singular Riemannian foliation with closed leaves and $sec(M)\geq k$, then the leaf space $M^*$ is an Alexandrov space with $curv(M^*)\geq k$ (with respect to the metric induced by the Hausdorff distance between the leaves in $M$).
\end{rmk}

A curve $c\colon [0,1] \to X$ is a \emph{geodesic} if the length of $c$ equals the distance $d(c(0),c(1))$. Two geodesics $c_1\colon [0,1]\to X$ and $c_2\colon [0,1]\to X$ emanating from a common fixed point $x\in X$ are said to be \emph{equivalent} if the angle between them is zero. The set $\wtilde{\Sigma}_{x}$ of these equivalence classes becomes a metric space by declaring the distance between two classes to be the angle formed between any two representatives of each class. The \emph{space of directions $\Sigma_{x}(X)$} at $x$ of $X$ is the metric completion of the space $\wtilde{\Sigma}_x$. The following is a well known collection of results which will be crucial for of our ``inductive'' proof of the main theorem.

\begin{theorem}[see \cite{BGP}]
Let $X$ be an Alexandrov space of dimension $n$. Then for any $x\in X$, the space of directions $\Sigma_x (X)$ is a compact Alexandrov space with curvature at least $1$, and of dimension $n-1$.
\end{theorem}

For leaf spaces of singular Riemannian foliations, we have 

\begin{prop}[see p.~4 in \cite{moreno2019}]\label{L: Space of directions isometric to quotient of holonomy and infinitesimal foliation}
The space of directions of the Alexandrov space $M^\ast$ at $p^\ast$, consists of geodesic directions and is isometric to $(\Sp_p^\perp/\FF_p)/\hol_{p}$.
\end{prop}

It is worth mentioning here that for a principal leaf $L_p\subset M$, the infinitesimal foliation $(\Sp_p^{\perp},\FF_p)$ is a foliation by points and the leaf holonomy $\Gamma_p$ is trivial. Thus for $p\in M$ contained in a principal leaf, the space of directions at $p^*\in M^*$ is $\Sigma_{p^*}\cong\Sp_p^{\perp}$.

From the description of singular Riemannian foliations on round spheres in \cite{Radeschi2012}, and the discussion in \cite[pp.~25--28]{morenophd}, we have

\begin{lemma}\label{L: leaves in a geodesic pointing in a regular direction are regular}
	Consider $(M,\FF)$ a closed singular Riemannian foliation. Fix $p\in M$ and consider $v\in\Sp^\perp_p$. Then  $\mathcal{L}_v$ is a regular leaf of $(\Sp_p^\perp,\FF_p)$ if and only if for $q = \exp_p(v)$, the leaf $L_q$ is a regular leaf. Moreover, if $q$ is in a regular leaf, then  for $t\in (0,1]$, the leaf $L_{\gamma(t)}$ is a regular leaf.
\end{lemma}

Moreover, it is true that all strata of $\FF$ (components of leaves of the same dimension) whose closure contains the leaf $L_p$ appear as strata of $(\mathbb{S}_p^{\perp},\FF_p)$. In particular, nearby leaves of the same dimension as $L_p$ appear as 0-dimensional leaves in $\FF_p$. If $\mathbb{S}_p^{\perp}$ is an $n$-dimensional sphere, we have the following splitting given by Radeschi in \cite{Radeschi2012}: \[(\mathbb{S}^n,\FF_p)\cong(\mathbb{S}^k,\FF_0)*(\mathbb{S}^{n-k-1},\FF_1),\] where $\FF_0$ is a foliation by points and $\FF_1$ is a foliation containing no point leaves. We refer to $(\mathbb{S}^{n-k-1},\FF_1)$ as the \textit{infinitesimal foliation normal to the stratum of $L_p$} and refer to the quotient $\mathbb{S}^{n-k-1}/\FF_1$ as the \textit{space of directions normal to the stratum of $L_p$}. Since $\mathbb{S}^k/\FF_0=\mathbb{S}^k$, spaces of directions normal to strata are where \textit{boundary} will be detected, as defined for Alexandrov spaces below:

\begin{definition}
Let $X$ be an Alexandrov space. The \textit{boundary} of $X$, denoted $\bdy(X)$ is defined inductively as  \[ \partial X := \{x\in X\mid \partial \Sigma_{x}(X) \neq \emptyset\}.\] Where we use the fact that spaces of directions are compact positively curved Alexandrov spaces with $\dim(\Sigma_{x}(X))=\dim(X)-1$ and the only such 1-dimensional Alexandrov spaces are circles or closed intervals (with diameter $\leq \pi$ in both cases).
\end{definition}




\section{Reductions of the foliation}

In this section we present the definition of  a pre-section for a closed singular Riemannian foliation, and lemmas that will be used in the proof of Theorem~\ref{Theorem: Core reduction}.

\subsection{Pre-sections}
Let us begin by giving the definition of a pre-section:
\begin{definition}\label{D: generalized section}
Let $(M, \FF)$ be a closed singular Riemannian foliation on a complete manifold. A connected embedded submanifold $N\subset M$ is a \emph{pre-section of $(M,\FF)$} if the following are satisfied: 
\begin{enumerate}[(A)]
	\item\label{D: generalized section A)} $N$ is complete, totally geodesic,
	\item\label{D: generalized section B)} $N$ intersects every leaf of $\FF$,
	\item\label{D: generalized section C)} for every point $p$ in $N\cap M_{\reg}$ we have $\nu_{p}(M, L_p) \subset T_p N$.
\end{enumerate}

We say that a pre-section $N\subset M$ is \textit{nontrivial} if $N$ is a proper submanifold and is not a single point.
\end{definition}

Note that condition \eqref{D: generalized section C)} implies that a pre-section intersects the regular leaves transversally.\\

Consider $\FF'$ to be partition of $N$ into the connected components of the intersections of $N$ with the leaves in $\FF$, i.e.\ the restriction of $\FF$ to $N$. We start by proving the following lemma for regular leafs of $\FF$.

\begin{lemma}\label{L: Closest point to regular leaf is in the normal sphere}
Consider $p\in N$ such that $L_p$ is a regular leaf of $\FF$. Assume $q\in N$ is a closest point in $L_q\cap N$ to $p$. Then $q$ is a closest point to $p$ in $L_q$ and $q\in \exp_p(\Sp^\perp_p)$.
\end{lemma}

\begin{proof}
Assume there exists $q'\in L_q$ different from $q$ which is a closest point to $p$ in $L_q$. From this it follows that $q' = \exp_p(v)$ for some $v\in \Sp_p^\perp$. Since $p$ is a regular leaf, then condition \eqref{D: generalized section C)} implies that  $\Sp_p^\perp\subset T_pN$. Thus the minimizing geodesic in $M$ joining $p$ to $q'$ given by $\gamma(t) = \exp_p(tv)$, is a geodesic in $N$. But this implies that $q'\in L_q\cap N$, and the distance from $q$ to $p$ is larger or equal than the distance from $q'$ to $p$. The distance can not be strictly larger, since this would contradict the fact that $q$ is a closest point to $p$ in $L_q\cap N$. Thus the distance from $q$ to $p$ in $M$ realizes the distance from $L_q$ to $p$. Therefore, there exists $v_0\in \Sp^\perp_p$ with $q = \exp_p(v_0)$.
\end{proof}

\begin{lemma}
Consider $(M,\FF)$ be a closed singular Riemannian foliation on a complete manifold. Let $N\subset M$ be a pre-section of $(M,\FF)$. Then, the partition $(N,\FF')$ is a singular Riemannian foliation with respect to the induced metric of $M$ on $N$.
\end{lemma}

\begin{proof}
Let $\{X_i\}$ be the family of vector fields in $M$ which span the tangent spaces of the leaves of $\FF$. Let $Y_i$ be the projection of $X_i$ to the tangent bundle of $N$. It is clear that for any leaf $L\in \FF$, the vector fields $Y_i$ span the tangent space of $L\cap N$. We also point out that for $p\in N$ we have $T_pL_p = T_p(L_p\cap N)\oplus \nu_p(L_p,L_p\cap N)$. Thus the vectors $Y_i$ 
 are also tangent to the leaves of $\FF$. 

We prove now that $\FF'$ gives a transnormal system (i.e.\ for any geodesic $\gamma\colon I \to N$ with $\gamma'(0)\perp L_{\gamma(0)}\cap N$, we have $\gamma'(t)\perp L_{\gamma(t)}\cap N$ for all $t\in I$). Note that because $N$ is totally geodesic, such a $\gamma$ is also a geodesic of $M$.

We first show that if $\gamma$ emanates orthogonally to $L_p\cap N\subset \FF'$ from a point $p\in N$ belonging to a regular leaf of $\FF$, then its intersections with connected components of the partition $\FF'$ are orthogonal. 

Let $p=\gamma(0)$, assume that $|\gamma'(0)| = 1$ and  $q=\gamma(1) = \expo_p(\gamma'(0))$ is in a regular leaf $L_q$ of $\FF$. Since $\gamma'(0)\in \nu_p(N,L_p\cap N)$, then $q$ is a closest point to $p$ in $L_q\cap N$. Thus by Lemma~\ref{L: Closest point to regular leaf is in the normal sphere} the point $q$ is a closest point in $L_q$ to $p$. Since $L_q$ is regular, and  then again by Lemma~\ref{L: Closest point to regular leaf is in the normal sphere} $p$ is the closest point to $q$ in $L_p\cap  N$. Thus $\gamma'(1)$ is perpendicular to $L_q\cap N$. Since by Lemma~\ref{L: leaves in a geodesic pointing in a regular direction are regular}, for all $t\in I$ the leaves $L_{\gamma(t)}$ are regular, the same conclusion holds. Thus $\gamma(t)$ intersects  the leaves $L_{\gamma(t)}\cap N$  of $\FF'$ perpendicularly.

Now let $q\in N$ be any sufficiently close point to $p$. Since $M_{\reg}$ is open and dense, then for the normal sphere $\pre{N}\Sp_p^\perp$ of $L_p\cap N$ at $p$ in $N$, the set  $M_{\reg}\cap \exp_p(\pre{N}\Sp^\perp_p)$ is open and dense in $\exp_p(\pre{N}\Sp^\perp_p)$. Thus there exists a convergent sequence $\{\exp_p(v_i)\}\subset M_{\reg}\cap \exp_p(\pre{N}\Sp^\perp_p)$ with $q$ as limit. By the previous paragraph, we have that $v_i\in \Sp_p^\perp$. Since $\exp_p$ is a local diffeomorphism, we conclude that the sequence $\{v_i\}$ converges to $v$. Since $\Sp_p^\perp$ is closed, it follows that $v\in \Sp_p^\perp$. Thus for the Riemannian metric $g$, we have that $g(\gamma'(t),w) = 0$ for all vectors $w\in T_{\gamma(t)}(L_{\gamma(t)})$. Since the vectors $Y_k$ are tangent to the leaves of $\FF$, we conclude that $g(\gamma'(t),Y_k)= 0$ for all indexes $k$.


In the case that $p$ belongs to singular leaf of $\FF$, a geodesic of $N$ emanating orthogonally from $L_p\cap N$ is not necessarily orthogonal to $L_p$, so we cannot immediately appeal to the transnormality of $\FF$ as above. Instead we `reverse' our frame of reference. Let $q$ be a nearby point along $\gamma$. Consider a sequence $\{p_i\}\in M_{reg}\cap N$ which converges to the singular point $p$. From each of these points, there is some geodesic $\gamma_i$ emanating orthogonally to $L_{p_i}\cap N$ minimizing the distance between $L_{p_i}\cap N$ and $L_q\cap N$ (hence meeting both orthogonally). These $\gamma_i$ converge to $\gamma$, from which it follows that $\gamma$ meets $L_q\cap N$ orthogonally, by continuity of the metric. This completes the proof of transnormality.
\end{proof}

\begin{theorem}
Consider $(M,\FF)$ be a closed singular Riemannian foliation on a complete manifold. Let $N\subset M$ be a pre-section. Then the inclusion $i\colon N \hookrightarrow M$ induces a discrete submetry $i\colon N/\FF'  \to M/\FF$, given by $i(L_p\cap N) = L_p$ 
\end{theorem}

\begin{proof}
Recall that the distance between $L_p\cap N $ and $L_q\cap N$ in $N/\FF'$ is given by $d_N (L_p\cap N ,L_q\cap N = \inf \{ d_M(x',y')\mid x'\in L_p\cap N, \  y'\in L_q \cap N\}$. On the other hand the distance between $L_p$ and $L_q$ in $M/\FF$, is equal to $\inf \{d_M(x,y)\mid x\in L_p, \  y\in L_q \}$. From this we see that in general for $p,q\in N$, we have $d_M(L_p,L_q) \leqslant d_{N}(L_p\cap N, L_q\cap N)$. Thus,  a ball of radius $r$ centered at $L_p\cap N$ in $N/\FF'$ gets mapped into the ball of of radius $r$ centered at $L_p$ in $M/\FF$.

To prove that $i\colon N/\FF'  \to M/\FF$ is a submetry, we have to prove that for any $r>0$ and any leaf $L_p$ of $\FF$, if $L_q$ is such that $d_M(L_q,L_p)<r$, then $d_{N}(L_q\cap N, L_p\cap N)< r$. That is the map from the ball of radius $r$ around $L_p\cap N$ in $N/\FF'$ to the ball of radius $r$ around $L_p$ in $M/\FF$ is onto. For this it is sufficient to prove this for any sufficiently small radius.

We will now prove that for $p\in M_{\reg}\cap N$, the map $i\colon N/\FF'  \to M/\FF$ is a local isometry for a sufficiently small ball around $L_p$. Fix $r>0$ with $r$ smaller than the injectivity radius of $M$ at $p$. Take $L_q$ such that $d_M(L_p,L_q)<r$, and let $\gamma$ be the minimizing geodesic between $L_p$ and $L_q$ starting at $p$; i.e. $\ell(\gamma) = d_M(L_p,L_q)$ and $\gamma(0) = p$. Then $\gamma$ is perpendicular to $L_p$ at $p\in M$. Since $N$ is a pre-section we have $\nu_p(M,L_p)\subset T_p N$, and since $N$ is totally geodesic, we conclude that $\gamma$ is a geodesic in $N$. Let $q'\in L_q\cap N$ be a closest point to $p$ in $L_q\cap N$. By Lemma~\ref{L: Closest point to regular leaf is in the normal sphere} we have that $q'$ is the closest point to $p$ in $L_q$. Thus we have
\[
	 d_N(L_q\cap N, L_p\cap N) = d_N(q',p) = d_M(q',p) = d_M(L_q,L_p)< r.
\]

Now consider $p\in N$ arbitrary, and take $r$ smaller that the injectivity radius of $M$ at $p$. Fix $L_q$ of $\FF$ such that $d_M(L_p,L_q)<r$. Let $\gamma\colon [0,1]\to M$ be the minimizing geodesic  of $M$ joining $L_p$ to $L_q$ starting at $p$. Since $M_{\reg}$ is dense, for any $n>2$ there exists $p_n$ such that for the middle point $\gamma(1/2)$ we have
\[
	d_M(p_n,\gamma(1/2))<\frac{r}{n}.
\]
By applying the triangle inequality and the fact that $d_M(\gamma(1/2), q) \leqslant r/2$ and $d_M(p,\gamma(1/2)) \leqslant r/2$ we conclude that:
\begin{align*}
	d_M(p_n,q) &\leqslant d_M(p_n,\gamma(1/2)) + d_M(\gamma(1/2),q) \leqslant r/n+r/2 < r;\\
	d_M(p,p_n) &\leqslant d_M(p,\gamma(1/2)) + d_M(\gamma(1/2), p_n) \leqslant r/n+r/2 < r.
\end{align*}
Thus we get,
\begin{align*}
		d_N(L_p\cap N, L_q\cap N) &\leqslant d_N(L_p\cap N, L_{p_n}\cap N) + d_N(L_{p_n}\cap N, L_{q}\cap N)\\ 
		&= d_M(L_p, L_{p_n}) + d_M(L_{p_n}, L_{q}) \leqslant \frac{2r}{n}+r.
\end{align*}
Consequently, by taking the limit when $n$ goes to infinity, we conclude that
\[
	d_N(L_p\cap N, L_q\cap N) \leqslant r.
\]

%
\end{proof}

\begin{theorem}[Slice Theorem] \label{S: Slice Theorem}
Let $(M,\FF)$ be a closed singular Riemannian foliation on a complete manifold containing a pre-section  $N\subset M$. For any $q\in N$, the space $V_q = \nu_q(M,L_q)\cap T_q N$ is a pre-section for the infinitesimal foliation $(\nu_q(M,L_q), \FF_q)$. 
\end{theorem}

\begin{proof}
Observe that $V_q$ is a linear subspace of $\nu_q (M, L_q)$, so it is totally geodesic and complete. Hence, it satisfies condition \eqref{D: generalized section A)} of Definition~\ref{D: generalized section}.

Since $N$ intersects each leaf  of $\FF$ and $N$ is complete, then $T_q N$ intersects each leaf of the infinitesimal foliation  $\FF_q$ at $q$, i.e.\ $V_q$ satisfies condition \eqref{D: generalized section B)} of  Definition \ref{D: generalized section}. 

It remains to prove that $V_q$ satisfies condition \eqref{D: generalized section C)} of  Definition \ref{D: generalized section}. Fix $v\in V_q$ such that $\mathcal{L}_v$ is a regular leaf of the infinitesimal  foliation $\FF_q$, and observe that, as in \cite{magata2009reductions}, the property \eqref{D: generalized section C)} is equivalent to $\nu_v (\nu_q (M, L_q),V_q) \subset T_v \mathcal{L}_v$. Since the infinitesimal foliation is invariant under homotheties, we may assume that $v$ is small enough, so that $p = \expo_q(v)$ is contained in a tubular neighborhood given by the slice theorem in \cite{MendesRadeschi2019}. 

Recall that $(\nu_q(M,L_q),\FF_q)$ is a singular Riemannian foliation  with respect to  the Euclidean metric (see \cite[Section~1.2]{Radeschi2012}), thus we fix this metric on $\nu_q(M,L_q)$. We observe that since $\nu_q(M,L_q)$ is a linear space, then we can identify $T_v \nu_q(M,L_q)$ with $\nu_q(M,L_q)$. Since $V_q$ is a linear subspace of $\nu_q(M,L_q)$, under this identification we identify $T_v V_q $ with $V_q$. Moreover, for the Euclidean metric we have the following splittings: $ T_v \nu_q(M,L_q) = T_v V_q\oplus \nu_v (\nu_q(M,L_q),V_q)$ and $\nu_q(M,L_q)  = V_q \oplus V_q^\perp$. Thus, we can identify $\nu_v (\nu_q(M,L_q),V_q)$ with $V_q^\perp$.

Now we consider the decomposition $T_q M = T_q N \oplus \nu_q(M,N)$ with respect to the foliated Riemannian metric of $M$ at $q$. Then
\begin{linenomath*}
\begin{align*}
	V_q^\perp &= V_q^\perp \cap T_q M = V_q^\perp \cap (T_q N \oplus \nu_q(M,N) )\\ &= (V_q^\perp \cap T_q N)\oplus (V_q^\perp \cap \nu_q(M,N)).
\end{align*}	
\end{linenomath*}
Consider $x\in V_q^\perp\cap T_q N$. Since $V_q^\perp\subset \nu_q(M,L_q)$, we conclude that $x\in T_q(N)\cap \nu_q(M,L_q)=V_q$. So $x\in V_q\cap V_q^\perp=\{0\}$ and therefore $x=0$. This implies that $V_q^\perp = V_q^\perp\cap \nu_q(M,N)\subset \nu_q(M,N)$. 

We consider $w\in \nu_v(\nu_q(M,L_q), V_q)$ arbitrary. Let $w'\in \nu_q(M,L_q)$ be the vector corresponding to $w$ under the identification of $T_v(\nu_q(M,L_q))$ with $\nu_q(M,L_q)$. Then by the previous paragraphs we have, $w'\in V_q^\perp\subset \nu_q(M,N)$.
Observe that since $v\in V_q\subset T_q N$ and $N$ is totally geodesic, then the geodesic $\alpha\colon I\to M$ given by $\alpha (s) = \exp_q(sv)$ is really a geodesic in $N$. Let $J(s)$ be the Jacobi field along $\alpha(s)$ determined by $J(0)=0$, and $J'(0) = w'$. By Lemma~4.1 in \cite{magata2009reductions} we have that for all $s\in I$, $J(s)\in \nu_{\alpha(s)}(M,N)$.  
We have by \cite[Chapter~IX, Thm~3.1]{Lang}:
\[
	D_{v}(\exp_q)(w) = J(1)\in\nu_p(M,N).
\]
Since $v\in V_q$ is such that $\mathcal{L}_v$ is a regular leaf of $\FF_q$, we have by Lemma~\ref{L: leaves in a geodesic pointing in a regular direction are regular} that $L_p$ is a regular leaf of $\FF$. By property \eqref{D: generalized section C)}, we have that $\nu_p(M,N)\subset T_p L_p$. Thus $D_{v}(\exp_q)(w)\in T_p L_p$. Moreover the exponential map of $M$ at $q$ induces a local diffeomorphism $\exp_q\colon \mathcal{L}_v\to L_p$. By considering the derivative at $v$, we have an isomorphism $D_v(\exp_q)\colon T_v\mathcal{L}_v\to T_pL_p$. Thus we conclude that $w\in T_v \mathcal{L}_v$ as desired.

%
%

\end{proof}

\begin{rmk}\label{R: Dimension of presection for infinitesimal foliation}
Since $V_q$ is a subspace of $\nu_q (M,L_q)$, we have $\dim (V_q) \leqslant \dim (\nu_q (M,L_q))$. This obvious statement provides a comparison of the relative codimensions of leaves in a pre-section with that of the associated `ambient' leaves. Namely, for a given leaf $L_q\cap N$ of the pre-section foliation $\FF'$, we have that $\codim(N, L_q\cap N)\leq \codim(M,L_q)$.
\end{rmk}

The technique employed to prove the main theorem will involve ``chasing the codimension drop'' and inductively reducing via pre-sections. For our argument, we will need to focus on the infinitesimal foliation normal to such a leaf's \textit{stratum} (the component of leaves of the same dimension containing that leaf). In the infinitesimal foliation, nearby leaves of the same dimension appear as point leaves and we have the splitting \[(V,\FF)=(V_0\times V_0^{\perp},\{pts\}\times \FF_{>0})\] where $V_0$ is the linear subspace of $V$ foliated by points (i.e. the tangent space to the stratum of the central leaf), $V_0^{\perp}$ is the orthogonal complement of $V_0$  with respect to the Euclidean metric (i.e. the normal space to the stratum) and $\FF_{>0}$ is a foliation whose only point leaf is the origin. If $(W,\FF')$ is a pre-section of $(V,\FF)$, then since $W$ intersects all leaves of $\FF$, we must have that $V_0\subset W$. With this, we have the splitting \[(W,\FF')=\left(V_0\times (V_0^{\perp}\cap W), \{pts\}\times (\FF_{>0})'\right)\]
where $(\FF_{>0})'$ is the partition of $V_0^{\perp}\cap W$ by its intersection with the leaves $\FF$.

\begin{lemma} \label{normal to stratum}
With the notation above, if $(W,\FF')$ is a pre-section of an infinitesimal foliation $(V,\FF)$, then $\left((V_0^{\perp}\cap W),(\FF_{>0})'\right)$ is a pre-section of $(V_0^{\perp},\FF_{>0})$.
\end{lemma}

\begin{proof}
Since $V_0^{\perp}\cap W$ is a linear subspace of $V_0^{\perp}$, it is totally geodesic and we have condition \eqref{D: generalized section A)}. Moreover, a leaf of $\FF$ is of the form $\{a\}\times L$, where $a\in V_0$ is a point leaf and $L\in \FF_{>0}$ and since $(W,\FF')$ is a pre-section of $(V,\FF)$, we have 
\begin{linenomath*}
\begin{align*}
W \cap ({\{a\}\times L}) &\neq \emptyset\\
\implies \left(V_0\times (V_0^{\perp}\cap W)\right) \cap (\{a\}\times L) &\neq\emptyset\\
\implies \left(V_0\cap \{a\}\right)\times \left((V_0^{\perp}\cap W)\cap L\right) &\neq \emptyset\\
\implies \{a\}\times (W_0^{\perp}\cap L)&\neq\emptyset\\
\end{align*}
\end{linenomath*}
Hence $V_0^{\perp}\cap W$ intersects every leaf of $\FF_{>0}$ and we have condition \eqref{D: generalized section B)}. Now let $L\in \FF_{>0}$ be a regular leaf through $p\in L\cap (V_0^{\perp}\cap W)$. Again, since $(W,\FF')$ is a pre-section of $(V,\FF)$, we have $\nu_{(a,p)}(V,\{a\}\times L)\subset T_{(a,p)}(W)$. Given the splittings above, we have 
\begin{linenomath*}
\begin{align*}
\nu_{(a,p)}\left(V_0\times V_0^{\perp},\{a\}\times L\right)&\subset T_{(a,p)}\left(V_0\times (V_0^{\perp}\cap W)\right)\\
\implies \nu_a(V_0,a)\times \nu_p(V_0^{\perp},L)&\subset T_a(V_0)\times T_p(V_0^{\perp}\cap W)
\end{align*}
\end{linenomath*}
and since $\nu_a(V_0,a)=T_a(V_0)$, it follows that $\nu_p(V_0^{\perp},L)\subset T_p(V_0^{\perp}\cap W)$, so condition \eqref{D: generalized section C)} is satisfied. 
\end{proof}

\begin{lemma} \label{restrict to spheres}
Let $(V,\FF)$ be an infinitesimal foliation and $(\Sp_V,\FF|_{\Sp_{V}})$ denote the singular Riemannian foliation given by its restriction to the round unit sphere in $V$. If $(W,\FF')$ is a pre-section of $(V,\FF)$, then $(\Sp_W,\FF'|_{\Sp_{W}})$ is a pre-section of $(\Sp_V,\FF|_{\Sp_{V}})$. 
\end{lemma}

\begin{proof}
Since $W$ is totally geodesic, and $V$ is a Euclidean space, we conclude that $W$ is a linear subspace of $V$. Thus the unit sphere $\Sp_W$ is a totally geodesic submanifold of $\mathbb{S}_V$, so we need only show that properties \eqref{D: generalized section B)} and \eqref{D: generalized section C)} are satisfied by $(\mathbb{S}_W,\FF'|)$. Since $V$ is an infinitesimal foliation, the leaves of $\FF$ are 
contained in distance spheres about the origin, and since $(W,\FF')$ is a pre-section of $(V,\FF)$, it follows that each distance sphere about the origin in $W$ intersects every leaf of the distance sphere of the same radius about the origin in $V$. Thus, property \eqref{D: generalized section B)} is satisfied.\\
For property \eqref{D: generalized section C)}, first note that because leaves of $(V,\FF)$ are contained in distance spheres about the origin, a regular leaf of $(\Sp_V,\FF|_{\Sp_{V}})$ is exactly a regular leaf of $(V,\FF)$ (i.e. $L_p\cap \Sp_V=L_p$). So let $L_p$ be such a leaf with $p\in W$. We wish to show that $\nu_p(\Sp_V,L_p\cap \Sp_V)\subset T_p(\Sp_W)$. Now 
\begin{linenomath*}
\begin{align*}
\nu_p(\Sp_V,L_p\cap \Sp_V) &= \nu_p(\Sp_V,L_p)\\
			&= \nu_p(V,L_p)\cap T_p(\Sp_V)\\
			&\subset T_p(W)\cap T_p(\Sp_V)\\
		 	&= T_p(W\cap \Sp_V)\\
		 	&= T_p(\Sp_W)
\end{align*}
\end{linenomath*}
where the third line uses that $\nu_p(V,L_p)\subset T_p(W)$ since $(W,\FF')$ is a pre-section of $(V,\FF)$. The fourth line follows from the fact that $W$ and $\Sp_{V}$ intersect transversally in $V$. 
\end{proof}

\subsection{Pre-sections in Positive Curvature}
\ \\
For foliations on positively curved manifolds, the existence of a nontrivial pre-section guarantees an infinitesimal reduction in the following sense

\begin{lemma}\label{L: codimension drops somewhere}
Let $(M,\FF)$ be a closed singular Riemannian foliation on a complete manifold with  positive sectional curvature. Let $N\subset M$ be a pre-section of $\FF$, and for $p\in N$ set $V_p = T_p N \cap \nu_p(M,L_p)$. Then there exists $q\in N$ such that $\dim (V_q ) < \dim (\nu_q (M,L_q))$. That is, $\codim(N,L_q\cap N)<\codim(M,L_q)$.
\end{lemma}

\begin{proof}
Assume that $\dim (V_q ) = \dim (\nu_q (M,L_q))$ for all $q\in N$. Since $V_q=\nu_q(M,L_q)\cap T_qN$, it follows that $\nu_q(M, L_q)\subset T_qN$. Since $N$ is totally geodesic, this means that all horizontal geodesics from $q$ belong to $N$. Thus, the dual leaf  $L^{\#}_q$ (see \cite{Wilking2007}) is contained in $N$. Since $M$ is positively curved, it follows from Theorem~\ref{T: Wilking one leaf} in \cite{Wilking2007}, that the dual leaf of $\FF$ through $p$ is equal to $M$. This implies that $N=M$, which is a contradiction. Thus $\dim (V_q ) < \dim (\nu_q (M,L_q))$ for some $q\in N$. 
\end{proof}

This is particularly important when restricting the infinitesimal foliation to its associated round unit sphere foliation. 

\section{Proof of Theorem~\ref{Theorem: Core reduction}}

With Theorem \ref{S: Slice Theorem} and Lemma \ref{L: codimension drops somewhere} 
have the  necessary ingredients to prove the main theorem:

\begin{proof}[Proof of Theorem~\ref{Theorem: Core reduction}]
Let $(N,\FF')$ be a nontrivial pre-section of $(M,\FF)$. We are assuming $M$ is positively curved, so by Lemma \ref{L: codimension drops somewhere}, there exists a point $q\in N$ (necessarily belonging to a singular leaf of $\FF$) such that $\codim(N, L_q\cap N)< \codim(M,L_q)$. From Theorem \ref{S: Slice Theorem}, we have that $V_q=\nu_q(M,L_q)\cap T_qN$ is a pre-section for the infinitesimal foliation $(\nu_q(M,L_p),\FF_q)$. By using Lemma \ref{normal to stratum}, we will focus on the foliation normal to the stratum of $L_q$ and its pre-section. By Lemma \ref{restrict to spheres}, we restrict this (normal to the stratum of $L_q$) infinitesimal foliation and its pre-section to their respective unit spheres and refer to them as $(M_1, \FF_1)$ and $(N_1,\FF'_1)$. Observe that $N_1$ is a proper submanifold of $M_1$. In particular, we have $\dim(M_1)<\dim(M)$ and $M_1$ is positively curved (it is a round sphere) and $\FF_1$ contains no point leaves.

If $(N_1,\FF'_1)$ is a nontrivial pre-section (equivalently, $(M_1, \FF_1)$) is not a single leaf foliation), then we can repeat this process at a point $q_2\in N_1$ where the relative codimension drops as in Lemma \ref{L: codimension drops somewhere} to form $(M_2,\FF_2)$ with pre-section $(N_2,\FF'_2)$. Now, we have \[\dim(M_2)<\dim(M_1)<\dim(M)\] and $M_2$ is positively curved.

For as long as $M_i$ is positively curved (it is a sphere, so this means simply that $\dim(M_i)>1$) and $(M_i,\FF_i)$ is not a single leaf foliation, this process of ``chasing the drop in codimension'' will continue. If we encounter $M_i=\mathbb{S}^1$, then since $\FF_i$ cannot contain point leaves, it must be that this is a single leaf foliation. Thus, this process necessarily ends with a single leaf foliation $(M_i,\FF_i)$. In this case, let $L_{q_i}\in \FF_{i-1}$ be the chosen leaf of $M_{i-1}$ whose relative codimension dropped. The fact that $(M_i,\FF_i)$ is a single leaf foliation means precisely that the space of directions normal to the stratum of $L_{q_i}\subset M_{i-1}$ is a single point. This implies that the this stratum forms a boundary face (see pg. 5 of \cite{moreno2019}) of the Alexandrov leaf space $M_{i-1}/\FF_{i-1}$. By the inductive definition of boundary for Alexandrov spaces, this implies that $\bdy(M_{i-2}/\FF_{i-2})\neq \emptyset$, which implies that $\bdy(M_{i-3}/\FF_{i-3})\neq \emptyset$, and ulitimately, that $\bdy(M/F)\neq \emptyset$
\end{proof}

\bibliographystyle{siam}
\bibliography{References}

\begin{thebibliography}{10}

\bibitem{Alexandrino2012}
{\sc M.~M. Alexandrino, R.~Briquet, and D.~T\"oben}, {\em Progress in the
  theory of singular {R}iemannian foliations}, Differential Geom. Appl., 31
  (2013), pp.~248--267.

\bibitem{BBI}
{\sc D.~Burago, Y.~Burago, and S.~Ivanov}, {\em A course in metric geometry},
  vol.~33 of Graduate Studies in Mathematics, American Mathematical Society,
  Providence, RI, 2001.

\bibitem{BGP}
{\sc Y.~Burago, M.~Gromov, and G.~Perel'man}, {\em A. {D}. alexandrov spaces
  with curvature bounded below}, Russian mathematical surveys, 47 (1992), p.~1.

\bibitem{Corro2019}
{\sc D.~Corro}, {\em {A}-{F}oliations of codimension two on compact
  simply-connected manifolds},
  \href{https://arxiv.org/pdf/1903.07191.pdf}{arXiv:1903.07191 [math.DG]},
  (2019).

\bibitem{galaz2015singular}
{\sc F.~Galaz-Garcia and M.~Radeschi}, {\em Singular {R}iemannian foliations
  and applications to positive and non-negative curvature}, J. Topol., 8
  (2015), pp.~603--620.

\bibitem{gorodski2004copolarity}
{\sc C.~Gorodski, C.~Olmos, and R.~Tojeiro}, {\em Copolarity of isometric
  actions}, Trans. Amer. Math. Soc., 356 (2004), pp.~1585--1608.

\bibitem{GroveMorenoPetersen2018}
{\sc K.~Grove, A.~Moreno, and P.~Petersen}, {\em The boundary conjecture for
  leaf spaces}, Ann. Inst. Fourier, 69 (2019), pp.~2941--2950.

\bibitem{grovesearlecore}
{\sc K.~Grove and C.~Searle}, {\em Global {$G$}-manifold reductions and
  resolutions}, Ann. Global Anal. Geom., 18 (2000), pp.~437--446.

\bibitem{Lang}
{\sc S.~Lang}, {\em Fundamentals of differential geometry}, vol.~191 of
  Graduate Texts in Mathematics, Springer-Verlag, New York, 1999.

\bibitem{magata2008reductions}
{\sc F.~Magata}, {\em Reductions, resolutions and the copolarity of isometric
  groups actions}, PhD thesis, Westf\"{a}lischen {W}ilhelms-{U}niversit\"{a}t
  {M}\"{u}nster, 2008.

\bibitem{magata2009reductions}
\leavevmode\vrule height 2pt depth -1.6pt width 23pt, {\em Reductions,
  resolutions and the copolarity of isometric group actions},
  \href{https://arxiv.org/pdf/0908.0183.pdf}{arXiv:0908.0183 [math.DG]},
  (2009).

\bibitem{MendesRadeschi2019}
{\sc R.~A.~E. Mendes and M.~Radeschi}, {\em A slice theorem for singular
  {R}iemannian foliations, with applications}, Trans. Amer. Math. Soc., 371
  (2019), pp.~4931--4949.

\bibitem{Molino}
{\sc P.~Molino}, {\em Riemannian foliations}, vol.~73 of Progress in
  Mathematics, Birkh\"auser Boston, Inc., Boston, MA, 1988.

\bibitem{morenophd}
{\sc A.~Moreno}, {\em Alexandrov Geometry of Leaf Spaces and Applications}, PhD
  thesis, University of Notre Dame, 2019.

\bibitem{moreno2019}
\leavevmode\vrule height 2pt depth -1.6pt width 23pt, {\em Point leaf maximal
  singular {R}iemannian foliations in positive curvature}, Differential Geom.
  Appl., 66 (2019), pp.~181--195.

\bibitem{Radeschi2012}
{\sc M.~Radeschi}, {\em Low dimensional {S}ingular {R}iemannian {F}oliations in
  spheres}, PhD thesis, University of Pennsylvania, 2012.

\bibitem{Radeschi2014}
\leavevmode\vrule height 2pt depth -1.6pt width 23pt, {\em Clifford algebras
  and new singular {R}iemannian foliations in spheres}, Geom. Funct. Anal., 24
  (2014), pp.~1660--1682.

\bibitem{wilking2006positively}
{\sc B.~Wilking}, {\em Positively curved manifolds with symmetry}, Ann. of
  Math. (2), 163 (2006), pp.~607--668.

\bibitem{Wilking2007}
\leavevmode\vrule height 2pt depth -1.6pt width 23pt, {\em A duality theorem
  for {R}iemannian foliations in nonnegative sectional curvature}, Geom. Funct.
  Anal., 17 (2007), pp.~1297--1320.

\end{thebibliography}

\end{document}